\documentclass[reqno,11pt]{amsart}
\usepackage{latexsym}
\usepackage{amsfonts}
\usepackage{amssymb}
\usepackage{mathrsfs}
\usepackage[all]{xy}
\usepackage{bbm}
\usepackage{xcolor}
\usepackage{ulem}
\usepackage[mathscr]{euscript}

\usepackage[english]{babel}
\usepackage{amsmath}
\usepackage{graphicx}
\usepackage[colorinlistoftodos]{todonotes}
\usepackage[colorlinks=true, allcolors=blue]{hyperref}
\usepackage{tikz-cd}
\usepackage{tikz}
\usepackage{pst-node}
\usepackage{mathtools}

\usepackage[T2A]{fontenc}

\theoremstyle{plain}
\newtheorem{thm}{Theorem}[section]

\newtheorem{lem}[thm]{Lemma}
\newtheorem{rem}[thm]{Remark}
\newtheorem{prop}[thm]{Proposition}
\newtheorem{cor}[thm]{Corollary}

\newtheorem{exmp}[thm]{Example}

\theoremstyle{definition}

\theoremstyle{remark}

\numberwithin{equation}{section}

\newcommand{\SL}{\operatorname{SL}}
\newcommand{\GL}{\operatorname{GL}}

\newcommand{\dist}{\operatorname{dist}}

\newcommand\di{Diophantine}

\newcommand{\R}{{\mathbb{R}}}
\newcommand{\Z}{{\mathbb{Z}}}
\newcommand{\N}{{\mathbb{N}}}
\newcommand{\D}{{\mathbb{D}}}

\newcommand\hd{Hausdorff dimension}
\newcommand{\uu}{{\bf u}}

\newcommand{\vr}{{\bf r}}

\newcommand{\ve}{{\bf e}}
\newcommand{\vp}{{\bf p}}
\newcommand{\vq}{{\bf q}}
\newcommand{\x}{{\bf x}}
\newcommand{\ignore}[1]{{}}

\newcommand\eq[2]{
\begin{equation}
\label{eq:#1}
{#2}
\end{equation}
}
\newcommand{\equ}[1]{\eqref{eq:#1}}

\title[Dirichlet-improvable pairs and arbitrary norms]{{Abundance  of Dirichlet-improvable pairs \\ with respect to arbitrary norms}}
\author{Dmitry Kleinbock and Anurag Rao}
\address{Brandeis University, Waltham MA
02454-9110 {\tt kleinboc@brandeis.edu}}
\address{Wesleyan University, Middletown CT 06459-0260
{\tt arao@wesleyan.edu}}

\begin{document}

\begin{abstract}
In the paper \cite{AS} of Akhunzhanov--Shatskov  the two-dimensional Dirichlet spectrum with respect to  Euclidean norm was defined. 
We consider an analogous definition for arbitrary norms on $\mathbb{R}^2$ and prove that, for each such norm, the set of Dirichlet improvable pairs {contains the set of badly approximable pairs, hence is hyperplane absolute winning.}
To prove this we make a careful study of some classical results in the geometry of numbers due to Chalk--Rogers and Mahler to establish a Haj\'{o}s--Minkowski type result for the critical locus of a cylinder.
%This shows that \new{for all norms  the set of Dirichlet improvable pairs contains the set of badly approximable pairs, hence is hyperplane absolute winning.
%For the remaining cases of norms, we use an orbit avoidance theorem in the space of lattices due to the first-named author with An and Guan.}
{As a corollary, using a recent result of the first named author with Mirzadeh \cite{KM}}, we conclude that for any norm on $\mathbb{R}^2$ the top of the Dirichlet spectrum is not an isolated point.
\end{abstract}
\thanks{D.K.\ was  supported by  NSF grant  DMS-1900560.}
\date{October 14, 2021}

\maketitle
%\large
\section{Introduction}

For a pair of real numbers $\x = (x_1,x_2)$, the authors of \cite{AS} considered the quantity
\begin{equation}\label{as-spectrum}
    c_2(\x) := \limsup_{t\to \infty}\ t  \left(\min_{%q \in [1,\dots,t]\cap \Z
    1\le q \le t} \dist_2(q\x, \Z^2)\right)^2
\end{equation}
where %$q$ is taken to be an integer and 
$\dist_2$ denotes the distance function with respect to the Euclidean norm. 
%\comm{Is it ok to write $[1,t]$ here instead?}
Equivalently, $c_2(\x)$ is the infimum of $c>0$ such that the system 
$$\begin{cases} \| q\x - \vp\|_2^2 &\le c/t\\ \qquad  |q|&\le t, \end{cases}$$
where $\|\cdot\|_2$ is the Euclidean norm on $\R^2$,
has a
nontrivial integer solution $(\vp,q)$ for all large enough $t>0$.

The main result  of \cite{AS} was a description of the so-called  {\textit{Dirichlet spectrum} with respect to the Euclidean norm on $\R^2$}: 
%for pairs of real numbers:
\begin{equation}\label{spectrum}
  \D_2:=  \left\lbrace c_2(\x) : \x \in \R^2 \right\rbrace =  \left[0,%\frac
    {2}/{\sqrt{3}}\right].
\end{equation}
Note that the  set-up described above is a Euclidean modification of the classical set-up of improving Dirichlet's Theorem initiated by Davenport and Schmidt  \cite{Davenport-Schmidt}.
There the distance $\dist_\infty$ with respect to the supremum norm  is used in place of $\dist_2$ in \eqref{as-spectrum}. It is worth pointing out that, to the best our knowledge, the complete structure of the supremum-norm analogue of \eqref{spectrum}, that is the set 
\begin{equation}\label{sup-spectrum}
  \D_\infty:=  \left\lbrace c_\infty(\x) : \x \in \R^2 \right\rbrace , \text{ where } c_\infty(\x) := \limsup_{t\to \infty}\ t\cdot \left(\min_{%q \in [1,\dots,t]\cap \Z
    1\le q \le t} \dist_\infty(q\x, \Z^2)\right)^2,\end{equation}
 is not clear. What is known though is that $\D_\infty$ is a subset of $[0,1]$ containing both endpoints, that  $1$ is not an isolated point of $\D_\infty$, and that the set of \textit{Dirichlet-improvable} pairs
\begin{equation*}\label{sup-improvable}
\mathbf{DI}_{\infty} :=     \left\lbrace \x \in \R^2:  c_2(\x) < %\frac
    1\right\rbrace
\end{equation*}
has measure zero but nevertheless it is quite big: namely \cite[Theorem 2]{Davenport-Schmidt}  it contains the set of \textit{badly approximable} pairs
\begin{equation*}\label{BA}
\mathbf{BA} :=     \left\lbrace \x \in \R^2:  \liminf_{t\to \infty}\  t  \left(\min_{%q \in [1,\dots,t]\cap \Z
    1\le q \le t} \dist (q\x, \Z^2)\right)^2 > 0\right\rbrace.
\end{equation*}
The latter set is known to be of full \hd; this was proved in \cite{S} by showing that $\mathbf{BA}$ is a winning set of Schmidt's game, and later upgraded to an even stronger hyperplane absolute winning property, see \cite{BFKRW}.

%The purpose of this paper is to establish a complementary result.
%Namely, that 
In this paper we show that $\mathbf{BA}$ is also a subset of the set of \textit{Euclidean Dirichlet-improvable} pairs
\begin{equation*}\label{euclidean-improvable}
\mathbf{DI}_{2} :=     \left\lbrace \x \in \R^2:  c_2(\x) < %\frac
    {2}/{\sqrt{3}}\right\rbrace.
\end{equation*}
%is %of measure zero but 
%thick, that is, it intersects %measure zero but has 
%any open set in a set of full Hausdorff dimension. 
%also contains $\mathbf{BA}$. 
Moreover, the same is true if the Euclidean norm is replaced by any norm. Namely, for a norm  $\nu$ on $\R^2$ and $\x \in \R^2$ let us define
\begin{equation}\label{spectrum-definition}
    c_\nu(\x) = \limsup_{t\to \infty}\ t\cdot \left( \min_{
    1\le q \le t
    %q \in [1,\dots,t]\cap \Z
    } 
    \dist_\nu (q\x, \Z^2)\right)^2 \text{ and } c_\nu :=\sup_{\x\in\R^2}
    c_\nu(\x),
\end{equation}
where %, again, $q$ varies over the integers and 
$\dist_\nu$ is the distance function induced by $\nu$.
Our main result is as follows.
\begin{thm}\label{main-theorem}
For any  norm  $\nu$ on $\R^2$, the set
\begin{equation*}\label{main-theorem-equation}
\mathbf{DI}_{\nu} :=    \left\lbrace \x \in \R^2 : c_\nu(\x) < c_\nu \right\rbrace 
\end{equation*}
%is an HAW set of measure zero.
contains $\mathbf{BA}$ (hence is hyperplane absolute winning).
% but thick.
\end{thm}

Note that it has already been proved in \cite{KR} that for any norm $\nu$ the Lebesgue measure of $\mathbf{DI}_{\nu}$ is zero.
\smallskip

\ignore{\begin{rem} \rm
For the definition of the hyperplane absolute winning %(HAW) 
property, see \cite[\S2]{BFKRW} or \cite[\S2.1]{AGK}.
HAW implies winning in the sense of Schmidt
\cite{S}, and this in turn implies thickness.
{The class of} HAW sets, like those which are winning, {is} closed under countable intersections.% and this fact will be used in what follows.
\end{rem}

The measure zero fact has already been established in \cite{KR}, so the actual claim of the theorem is that the set is %thick.
HAW.
Further, this theorem is known to hold for the supremum norm $\|\cdot\|_\infty$: %Proposition \ref{AS-to-KR} below and 
it is a theorem of Davenport and Schmidt \cite[Theorem 2]{Davenport-Schmidt} that the set of \textit{badly approximable} $\x$ is contained in %the set of equation \eqref{main-theorem-equation}
$\mathbf{DI}_{\infty}$, and %\cite[Theorem 1]{Schmidt-BA} shows that badly approximable $\x$ form a thick set.
it is known \cite[Theorem 2.5]{BFKRW} that badly approximable vectors form an HAW set.

\smallskip}

%Following \cite{AS} 
%Similarly to \eqref{spectrum}, one can define

{Theorem \ref{main-theorem} can be used to derive a corollary concerning
 the  %(two-dimensional) 
 \textit{Dirichlet spectrum} with respect to an arbitrary norm $\nu$ on $\R^2$, defined similarly to \eqref{spectrum} and  \eqref{sup-spectrum} as
$$\D_\nu := \left\{c_\nu(\x) : \x\in\R^2\right\} . $$ It is %always 
a 
%closed 
subset of $[0,c_\nu]$ containing both of the endpoints. 
\begin{cor}\label{cor-spectrum} $c_\nu$ is an accumulation point of $\D_\nu $. 
\end{cor}}

The structure of this paper is as follows.
In the subsequent section we show that the set %in equation \eqref{main-theorem-equation} 
$\mathbf{DI}_{\nu}$ can be identified with a set of three-dimensional unimodular lattices having a certain dynamical property.
This property is that of avoiding a certain compact subset $\mathcal{L}_\eta$ of the space of all unimodular lattices in $\R^3$ under a diagonal flow. More precisely, $\mathcal{L}_\eta$ is  the \textit{critical locus} (see \S\ref{lattices}) of the norm $\eta$ on $\R^3$ given by 
\eq{norm}{  \eta(x_1,x_2,x_3) := \max\{\nu(x_1,x_2), |x_3|\}.}  
The next {two sections are} then justifiably spent on a detailed study of the work \cite{CR, CR-cor} of Chalk--Rogers in order to establish a structure theorem for these critical loci.
The resulting structure allows us to 
apply the argument implicitly contained in the paper of Davenport and Schmidt
%use a topological argument as well as a differential-geometric theorem of the first-named author with An and Guan 
to conclude the proof of Theorem \ref{main-theorem} in \S\ref{completion}. 
Corollary \ref{cor-spectrum} is then derived from Theorem \ref{main-theorem} with the help of a recent result---a solution of a special case of the Dimension Drop Conjecture---due to the first-named author and Mirzadeh \cite{KM}. Several  open questions are mentioned at the very end of the paper.

\subsection*{Acknowledgements}
We wish to thank Jinpeng An  for making an observation that helped us strengthen  our main result, and Nikolay Moshchevitin  for bringing the work \cite{AS} to our attention and for asking a question that led to Corollary~\ref{cor-spectrum}.

\section{{A} dynamical restatement}\label{lattices}

As promised, we describe how Dirichlet-improvability is related to dynamics and how the constant $c_\nu$ of equation \eqref{spectrum-definition} is related to critical lattices.
This is a special case of the uniform Diophantine approximation problem discussed in \cite{KR}.
{For  $n\in \N$ and a norm $\eta$ in $\R^n$, let $\Delta_\eta$ be the infimum covolume over all lattices in $\R^n$ intersecting the unit norm ball $B_\eta(1)\subset \R^n$ trivially.}
This is called the \textit{critical determinant} of the norm $\eta$.
The constant $\Delta_\eta$ is positive and is actually attained by a lattice.
These facts follow from the Minkowski convex body theorem and Mahler's compactness criterion respectively. 

For us it will be convenient to work with %We denote by $X$  the space of %three-dimensional 
\textit{unimodular} (covolume one) lattices; we will denote by  {$X_n$ the space of all such lattices  {in $\R^n$}. Clearly $X_n$ can be 
identified with the homogeneous space $\SL_n(\R)/\SL_n(\Z)$ via $g\mapsto g\Z^n$.
Then one can define the \textit{critical radius} of the norm $\eta$ as follows:
 $$
 r_\eta := \sup\left\{r:
 %\begin{aligned}\  
 \Lambda\cap B_\eta(r) = \{0\}%\\
  \text{ for some }\Lambda\in X_n%\ \end{aligned}
  \right\}. $$
It is clear that  we have $ r_\eta = \Delta_\eta^{-1/n}$.}

Moreover, for $r > 0$ let us define
\begin{equation*}
    \mathcal{K}_\eta(r) := \big\lbrace\Lambda \in {X_n} : \Lambda \cap B_\eta\left(%\frac
    {r}%{\Delta_\eta^{1/3}} 
    \right)  = \{0\}\big\rbrace.
\end{equation*}
%It follows from the definition that 
Obviously we have the %implication
containment
\begin{equation*}
    r_1 < r_2 \implies \mathcal{K}_\eta(r_1) \supset \mathcal{K}_\eta(r_2),
\end{equation*}
%The critical determinant's role here is to ensure that
and it follows from the definition of $
 r_\eta$ that \eq{empty}{
    r_\eta < r \implies \mathcal{K}_\eta(r) = \varnothing,
}  and  \eq{nonempty}{
  r<r_\eta \implies \mathcal{K}_\eta(r) %\neq \varnothing
    \text{ has non-empty interior}.
}

Thus, %for $0<r<r_\eta$, the family 
$\{\mathcal{K}_\eta(r): 0<r<r_\eta\}$ is a family of compact neighborhoods of the set $$\mathcal{L}_\eta:=\mathcal{K}_\eta(r_\eta),$$ which is called the \textit{critical locus} of $\eta$.

%Given a continuous nonincreasing approximation function $\psi : \R_{>0} \to \R_{>0}$,
%a pair $\x=(\alpha,\beta)$ is said to be $(\psi,\eta)$-Dirichlet if, for all sufficiently large $t$, 
%\begin{equation}\label{psi-dirichlet-definition}
% \left[ {\begin{array}{ccc}
%   1 & 0 & \alpha \\
%   0 & 1 & \beta \\
%   0 & 0 & 1
%  \end{array} } \right]\cdot\Z^3 \bigcap \left[ {\begin{array}{ccc}
 %  \psi(t) & 0 & 0 \\
%   0 & \psi(t) & 0 \\
%   0 & 0 & t
%  \end{array} } \right]
%  B_\eta\left(\frac{1}{{\Delta_{\eta}}^{{1/3}}}\right) \neq \{0\}.
%\end{equation}
\smallskip

{Now let us specialize to $n=3$.} For $\x\in\R^2$ let us define
%ease of notation we write
\begin{equation}\label{uA}
    \Lambda_\x := u_\x\cdot\Z^3\in {X_3}, \text{ where } u_\x := \left[ {\begin{array}{ccc}
   1 & 0 & x_1 \\
   0 & 1 & x_2 \\
   0 & 0 & 1
  \end{array} } \right] \in \SL_3(\R).
\end{equation}
%The set of $(\psi,\eta)$-Dirichlet tuples is written as $D_\eta(\psi)$.
%The lattice $\Lambda_\x$ is an element of the space 
The %diagonal 
flow of interest here is given by the one-parameter subgroup
\begin{equation}
    a_s:=\left[ {\begin{array}{ccc}
   e^{s/2} & 0 & 0 \\
   0 & e^{s/2} & 0 \\
   0 & 0 & e^{-s}
  \end{array} } \right].
\end{equation}
%By choosing $r$ to be the unique function satisfying
%\begin{equation}\label{Dani-r}
%    r\left( \frac{2}{3}\ln\frac{t}{\psi(t)}\right) = t^{1/3}\psi(t)^{2/3},
%\end{equation}
%and by defining 
%\begin{equation}
%    \mathcal{K}_\eta(r) := \left\lbrace\Lambda \in X : \Lambda \cap B_\eta\left(\frac{r}{\Delta_\eta^{1/3}} \right)  = \{0\}\right\rbrace,
%\end{equation}
%we get the following dynamical property equivalent to the condition of equation \eqref{psi-dirichlet-definition}.
%\begin{prop}\label{dynamical}
%$\x = (\alpha,\beta)$ is $(\psi,\eta)$-Dirichlet if and only if \begin{equation}\label{dani-cor}
%    a_s\Lambda_\x \notin \mathcal{K}_\eta(r(s))
%\end{equation}
%for all sufficiently large $s$. $\square$
%\end{prop}

%And moreover, the condition of equation \eqref{dani-cor} states the $(\psi,\eta)$-Dirichlet property as a property of the orbit of $\Lambda_\x$ missing a family of neighbourhoods eventually.
Recall that we started with a norm $\nu$ on $\R^2$, and let us take $\eta$ of the form \equ{norm}. 
Then a standard argument usually referred to as the \textit{Dani correspondence} gives the next proposition. 
It is a special case of a general correspondence described in \cite[Proposition 2.1]{KR}; we include the elementary proof for the sake of keeping the paper self-contained.

\begin{prop}\label{dynamical}
For  $\nu$ and $\eta$ related via \equ{norm} and $c  > 0$ the following are equivalent:
\begin{itemize}
\item[{\rm (i)}]
 $c_\nu(\x) < c  $;
 \item[{\rm (ii)}] there is an $r < c^{1/3}$  such that for all sufficiently large $s$, $a_s\Lambda_\x \notin  \mathcal{K}_\eta(r)$.
 \end{itemize}
 \end{prop}
 \begin{proof}
 Say $\rm (i)$ holds. 
 Then for any $c_\nu(\x) < d < c$, the system 
 $$\begin{cases} \nu( q\x - \vp) &\leq (d/t)^{1/2}\\ \qquad  |q|&\le t \end{cases}$$
 has a nontrivial integer solution $(\vp, q)$ for all sufficiently large $t$.
 Equivalently, for sufficiently large $t$,
 $$
 \Lambda_\x \cap \left[ {\begin{array}{ccc}
   (d/t)^{1/2} & 0 & 0 \\
   0 & (d/t)^{1/2} & 0 \\
   0 & 0 & t
  \end{array} } \right] \cdot \overline{B_\eta(1)} \neq \{0\}.
 $$
 By substituting $\ln\frac{t}{d^{1/3}}$ for $s$, we see that for all sufficiently large $s$,
 $$
 a_s\Lambda_\x \cap   \overline{B_\eta(d^{1/3})} \neq \{0\}.
 $$
 Choosing $d< r^3 < c$ we see that $\rm (ii)$ is true.
 
 Now say that $\rm (ii)$ holds.
 By defining $t$ via the equation $\ln \frac{t}{r} = s$, we may trace backwards in the implications above to see that the system 
  $$\begin{cases} \nu( q\x - \vp) & < (r^3/t)^{1/2}\\ \qquad  |q|& <\, t \end{cases}$$
  has a nontrivial solution for all sufficiently large $t$.
  Thus we have that $c_\nu(\x) \leq r^3 < c$, and hence $\rm (i)$ holds.
 \end{proof}
 
In view of \equ{empty} we have $c_\nu(\x) \le r_\eta^3$ for all $\x$. Moreover, an application of the above proposition with $c = r_\eta^3$, in view of \equ{nonempty} and the ergodicity of the $a_s$-action on ${X_3}$, shows that $c_\nu(\x)  = r_\eta^3$ for Lebesgue almost all $\x$. 
See \cite[Proposition 3.2]{KR} for additional details.
It follows that $c_\nu = r_\eta^3 = \Delta_\eta^{-1}$. {This last fact and another application of  Proposition \ref{dynamical} show that}

 %By using Proposition \ref{dynamical} and the discussion that follows it, we finally have
\begin{cor}\label{miss-locus}
%A vector 
$\x \in\mathbf{DI}_{\nu} $ 
%belongs to 
%\begin{equation}\mathbf{DI}_{\nu} := \bigcup_{c<1} D_\eta\left(c\psi_{1/2}\right)
%\end{equation}
if and only if the forward $a_s$-orbit of $\Lambda_\x$ eventually avoids an open neighborhood of $\mathcal{L}_\eta$; that is,
\eq{DI}
    %\x \in \mathbf{DI}_{\nu}  \iff
     {\overline{\left\lbrace a_s\Lambda_\x : s\ge s_0 \right\rbrace} \cap \mathcal{L}_\eta = \varnothing\text{ for some } s_0 >0.}
%$\square$
\end{cor}

Thus our main result is reduced to proving that {any $\x\in\mathbf{BA}$ satisfies \equ{DI}}.

\section{The critical locus of a cylinder: results of Chalk--Rogers and Mahler}\label{locus-of-cylinder}

Let us now slightly change our definitions of \textit{critical determinant} and \textit{critical locus} %change slightly
so that they now refer to objects corresponding to a convex symmetric bounded domain rather than a norm. %\comm{My changes stop here.}

For a  convex symmetric {bounded} domain $B\subset \R^n$, we say a lattice $\Lambda$ in $\R^n$ is \textit{$B$-admissible} if 
$$
B\cap \Lambda = \{0\}.
$$
We denote {by $d(\Lambda)$ the covolume of a lattice $\Lambda$, and by} $\Delta(B)$ the infimum covolume over all $B$-admissible lattices and refer to it as the \textit{critical determinant} of $B$.
The (nonempty compact) set $\mathcal{L}(B)$ denotes the set of $B$-admissible lattices having covolume exactly $\Delta(B)$, and we refer to it as the \textit{critical locus} of $B$. {To reconcile it with the previous notation, if $B$ is the unit ball with respect to a norm $\eta$, then $\mathcal{L}_\eta = \frac1{\Delta(B)^{1/n}}\mathcal{L}(B)$ consists of lattices in $\mathcal{L}(B)$ scaled so that their covolume becomes equal to $1$.}
Our goal is to {prove} a theorem which gives the structure of {$\mathcal{L}(B)$ when $B$ is a} cylindrical convex symmetric {bounded} domain in $\R^3$.

Let $B$ be a convex symmetric bounded domain in $\R^2$.
Let $\nu$ be the Minkowski norm associated to $B$.
Let ${C = C_B}$ be {defined as} the cylinder $B\times [-1,1]$.
We can also write $C$ as
\begin{equation}\label{cylinder-defn}
    C = \{(x_1,x_2,x_3) : \nu(x_1,x_2) <1,\ |{x_3}|<1\}.
\end{equation}
We fix notation ${\pi_{+}}$, ${\pi_{-}}$ for the projections from $\R^3$ to the  {$x_3$-axis and the $(x_1,x_2)$-plane} respectively.

It is a well-known fact that \eq{equaldets}{\Delta(B) = \Delta(C_B).}
In fact, it is very well known; around the $1940$s three independent proofs appeared. 
Mahler in \cite{Ma1} gave a proof for the cylinder over the disc, which was based on a (non-lattice) packing result for discs in a polgyon.
This packing result admits a generalization (due to Fejes-T\'oth) which settles the issue for arbitrary cylinders.
See \cite[\S IX.5.4]{Ca} for the theorem of Fejes-T\'oth as well as the following section for the application to cylinders.
Another argument was given by Yeh in \cite{Y} and was based on results of Minkowski on the configuration of critical lattices for three-dimensional domains.

The argument we discuss here is given by Chalk and Rogers in \cite{CR}. 
It so happens that this argument is especially convenient for making a precise description of $\mathcal{L}(C)$ in terms of $\mathcal{L}(B)$.
We need to make a preparatory study of the argument therein before giving the proof of the following theorem.
%In order to state this theorem, we need to make use of the notion of an irreducible domain.
%The domain $K$ is said to be irreducible if no properly contained convex symmetric $H \subset K$ has the same critical determinant as $K$.
%For a discussion of the critical locus $\mathcal{L}(K)$ for such irreducible domains in $\R^2$, see the paper \cite{KR}.

%We remark that this theorem was probably known to Mahler; see the first footnote in \cite[page 17]{Ma1} as well as the remark to Lemma $1$ in \cite{CR}.

\begin{thm}\label{critical-pieces}
Let $B \subset \R^2$ be a  convex symmetric bounded domain which is not a parallelogram, and let ${C = C_B}$.
%In the space of lattices $X_3$, 
Then the critical locus $\mathcal{L}(C)$ is  the union of the two sets
    \begin{equation}\label{critical-piece-21}
        \left\lbrace
        \left[ {\begin{array}{ccc}
   1 & 0 & 0 \\
   0 & 1 & 0 \\
   * & * & 1
  \end{array} } \right]
        \left[ \begin{array}{c | c} 
      M &  \begin{array}{c} 0 \\ 0 \end{array} \\ 
      \hline
      \begin{array}{c c} 0 & 0 \end{array} & 1 
     \end{array} \right] \Z^3 : M\Z^2 \in \mathcal{L}(B)
     \right\rbrace 
    \end{equation}
    and
    
    \begin{equation}\label{critical-piece-12}
        \left\lbrace
        \left[ {\begin{array}{ccc}
   1 & 0 & * \\
   0 & 1 & * \\
   0 & 0 & 1
  \end{array} } \right]
        \left[ \begin{array}{c | c} 
      M &  \begin{array}{c} 0 \\ 0 \end{array} \\ 
      \hline
      \begin{array}{c c} 0 & 0 \end{array} & 1 
     \end{array} \right] \Z^3 : M\Z^2 \in \mathcal{L}(B)
     \right\rbrace.
    \end{equation}
\end{thm}
The critical locus for cylinders over parallelograms and, more generally, for parallelopideds in $\R^n$ is the content of the Haj\'{o}s--Minkowski theorem.
See \cite{H} for the proof. 
 {A more expository treatment is given in \cite{SS}}.
\begin{thm}\label{Hajos}
If $C$ is the unit ball for the supremum norm on $\R^3$, the critical locus $\mathcal{L}(C)$ is the union of manifolds
\begin{equation}\label{Hajos-Minkowski-locus}
        P\left\lbrace
        \left[ {\begin{array}{ccc}
   1 & * & * \\
   0 & 1 & * \\
   0 & 0 & 1
  \end{array} } \right] \Z^3 
     \right\rbrace
    \end{equation}
    over all $3\times 3$ permutation matrices $P$. 
    %$\square$
\end{thm}

The following lemma appears as Lemma $1$ in \cite{CR}.
We write out the entire proof since we need an explicit form of the deformation, {see \eqref{shear?}  below,} appearing therein.
%See equations \eqref{shear-equation-lemma1.1} and \eqref{shear?} for this written out.
\begin{lem}\label{lemma1-CR} {Let $C$ be as in \eqref{cylinder-defn}.
Then,} given any critical lattice $\Lambda$ of $C$, there is a path of $C$-critical lattices $\Lambda(t)$ with $\Lambda(0) = \Lambda$, and with $\Lambda(1)$ having three linearly independent points satisfying
\begin{equation}\label{3points}
    \nu\big({\pi_{-}}(*)\big) \leq 1,\ |{\pi_{+}}(*)| = 1.
\end{equation}
In particular, three independent lattice points on the top of the cylinder can be found.
Moreover, this path is given by 
%\begin{equation}\label{shear-equation-lemma1.1}
   $\{ L_t\Lambda\}$,
%\end{equation}
where $L_t$ is a smooth family of lower triangular unipotents having $(2,1)$ entry zero.
\end{lem}
\begin{proof}
Let $\Lambda$ be $C$-critical. Then $\Lambda$ must contain three independent points on the boundary $\partial C$:
for if the points of $\Lambda \cap \partial C$ were contained in a plane, \cite[Theorem 7.8]{L} would enable us to construct an admissible lattice of smaller covolume, a contradiction.
So, let 
$\vp_1, \vp_2, \vp_3$
be any three independent points in $\Lambda \cap \partial C$.
Let $n$ be the number of linearly independent points of $\Lambda \cap \partial C$ satisfying \eqref{3points}. % \comm{Do you identify $-\vp$ with $\vp$ here? No.} 
{Clearly}
$n$ cannot be $0$, for otherwise it would permit a contraction in the vertical direction. 
In each of the cases $n= 1, 2$, we show that a shear can be applied to obtain a lattice $\Lambda'$ with $3$ independent points on $\partial C$ satisfying \eqref{3points}.

Say $n=1$. 
Among the three independent points $\vp_i$ in $\Lambda \cap \partial C$, we can assume without loss of generality that $\vp_1$ satisfies \eqref{3points} and that $\vp_2, \vp_3$ satisfy
\begin{equation}\label{sides-of-cylinder}
    \nu\big({\pi_{-}}(*)\big) = 1,\ |{\pi_{+}}(*)| < 1.
\end{equation}
Let $\vq_1,\ \vq_2,\ \vq_3$ be any basis for $\Lambda$. 
This gives us an integral $3\times 3$ matrix $M$ with
$$
\left[ {\begin{array}{ccc}
   \vp_1 & \vp_2 & \vp_3 \\
  \end{array} } \right] 
  = \left[ {\begin{array}{ccc}
   \vq_1 & \vq_2 & \vq_3 \\
  \end{array} } \right]M
$$
from which we can write
\begin{equation*}\label{Lambda-in-terms-of-3points}
\Lambda 
  = \left[ {\begin{array}{ccc}
   \vq_1 & \vq_2 & \vq_3 \\
  \end{array} } \right] \Z^3
  = \left[ {\begin{array}{ccc}
   \vp_1 & \vp_2 & \vp_3 \\
  \end{array} } \right] M^{-1} \Z^3.
\end{equation*}
Consider the continuous family of lattices
\begin{equation}\label{parameterization-of-continuous-family}
    \Lambda(t) = \left[ {\begin{array}{ccc}
   \vp_1 & \vp_2 & \vp_3 + t\ve_3 \\
  \end{array} } \right] M^{-1} \Z^3
\end{equation}
where $\ve_3$ is the third standard basis vector in $\R^3$.
 {Writing $P = \left[ {\begin{array}{ccc}
   \vp_1 & \vp_2 & \vp_3 \\
  \end{array} } \right]$ and $Q = \left[ {\begin{array}{ccc}
   \vq_1 & \vq_2 & \vq_3 \\
  \end{array} } \right]$}, we see that $M^{-1}= P^{-1}Q$, whence the lattice $\Lambda(t)$ can actually be written as
\begin{equation}\label{shear?}
\begin{split}
    \Lambda(t) &= 
    \left[ {\begin{array}{ccc}
   \vp_1 & \vp_2 & \vp_3 + t\ve_3 \\
  \end{array} } \right] P^{-1} Q\Z^3 \\
  &= 
  \left[ {\begin{array}{ccc}
   1 & 0 & 0 \\
   0 & 1 & 0 \\
   t(P^{-1})_{31} & t(P^{-1})_{32} & 1+t(P^{-1})_{33}
  \end{array} } \right]\Lambda,
  \end{split}
\end{equation}
where the constants on the bottom row indicate the corresponding entries in the matrix $P^{-1}$.
Without loss of generality, we assume that the covolume function {$$d\big(\Lambda(t)\big) = \left|1+t(P^{-1})_{33}\right|d(\Lambda)$$ has nonpositive derivative at $0$, i.e.\ $(P^{-1})_{33} \le 0$}. 
(If this is not the case, the argument continues by considering $-t$ instead of $t$.)
Let $\tau$ be {the largest  (possibly infinite)} real {number} such that
\begin{equation*}
0\leq t \leq \tau \text{ implies } \Lambda(t) \text{ is $C$-admissible}.
\end{equation*}
%Moreover, we can assume $\tau$ is the largest (possibly infinite) real having this property.
By the assumption on the derivative of the covolume, %and the fact that $d\big(\Lambda(t)\big)$ is piecewise linear \comm{(why piecewise? isn't it always equal to $( 1+t(P^{-1})_{33})d(\Lambda)$? Because there should be an absolute value sign here?)}, 
$\Lambda(t)$ must actually be $C$-critical for all $0\leq t\leq \tau$.  
%\comm{Sorry, I don't understand the argument here, either explain better or let's discuss on zoom at some point. For $0\leq t \leq \tau$, $\Lambda(t)$ is $C$-admissible and has $d(\Lambda(t)) \leq d(\Lambda(0))$ where $\Lambda(0)$ itself is $C$-critical}.
So, the covolume function is constant in this {(possibly singleton or even infinite)} interval. 
%\comm{This sentence probably should have come before the previous conclusion}.
We consider cases wherein $\tau$ is either infinite or not.
\begin{enumerate}
    \item[(a)] $\tau = +\infty$. Every $\Lambda(t)$ with $t\geq 0$ is $C$-critical.
    Note, in this case, since the covolume is constant, we must have $(P^{-1})_{33}=0$ in \eqref{shear?}.
    Choosing $t_0$ appropriately, we have $\vp_1$ and $\vp_3 + t_0\ve_3$ (which belong to $\Lambda(t)$) satisfying \eqref{3points}.
    The crucial observation here is that 
    \begin{equation}\label{shear-case-tau=infinity}
    \Lambda(t_0) = L(t_0)\Lambda
    \end{equation} where $L(t)$ is a curve in the lower triangular unipotents with $(2,1)$ entry $0$.
    \smallskip 
    
    \item[(b)] $\tau < \infty$. Consider $\Lambda(\tau)$, which is $C$-critical.
    Suppose we have no lattice point of $\Lambda(\tau)$ satisfying \eqref{3points} and {not proportional to} $\vp_1$.
    Choose $\rho>0$ so that every lattice point in $\Lambda(\tau)$ $\rho$-close (in some fixed norm $\|\cdot\|$) to $C$ is actually in the closure $\overline{C}$.
    Write 
    $$
    \Lambda(\tau) = N(\tau)\Z^3,\ \ \Lambda(\tau + \varepsilon) = N(\tau+\varepsilon)\Z^3
    $$
    where $N(\cdot)$ {is the product} of matrices appearing in the formula \eqref{parameterization-of-continuous-family}.
    Let $\delta>0$ be such that for all $\varepsilon \in [0,\delta]$, we have that
    $d\big(\Lambda(\tau+\varepsilon)\big) \geq \Delta(C)/2$ and that
    \begin{equation}\label{delta-small}
    \left(N(\tau+\varepsilon)\uu \in C \text{ with } \uu\in \Z^3\right) \ \text{ implies }\ 
    \left(\|N(\tau+\varepsilon)\uu-N(\tau)\uu \| \leq \rho\right).
    \end{equation}
    Now if we have $\varepsilon$ in this range, and a nonzero point $N(\tau+\varepsilon)\uu$ lying in $C$, then %equation
     \eqref{delta-small} and the definition of $\rho$ show that $N(\tau)\uu \in \overline{C}$.
    Since $\Lambda(\tau)$ is $C$-admissible, $N(\tau)\uu$ is in $\partial C$.
    By the assumption on $\Lambda(\tau)$ at the start of this case, we see that 
    \begin{equation}\label{change-for-n=2}
   \text{either } {\pi_{+}}\big(N(\tau)\uu\big) < 1, \text{ or } N(\tau)\uu = \pm \vp_1.
    \end{equation}
    The second eventuality is prevented since, by definition of $N(\cdot)$, we would then necessarily have $N(\tau +\varepsilon){\uu} = \pm \vp_1$, contrary to $N(\tau +\varepsilon){\uu} $ lying in $C$.
    Thus we are in the first case where ${\pi_{+}}\big(N(\tau) \uu\big) <1$, and this implies that $\nu\big({\pi_{-}}(N(\tau)\uu)\big)=1$.
    This is again a contradiction, since ${\pi_{-}}\big(N(\tau +\varepsilon)\uu\big) = {\pi_{-}}\big(N(\tau)\uu\big)$, and $N(\tau + \varepsilon){\uu} $ was presumed to be in $C$.
    
    What this shows is that $\Lambda(\tau+\varepsilon)$ must be $C$-admissible for every $\varepsilon$ in $[0,\delta]$, which is again incompatible with the definition of $\tau$. 
    Thus $\Lambda(\tau)$ must have more than $n=1$ independent points satisfying \eqref{3points}.
    Moreover, just as in equation \eqref{shear-case-tau=infinity}, we have \begin{equation*}\label{shear-case-tau=finite}
        \Lambda(\tau) = L(\tau)\Lambda ,
    \end{equation*} where $L(t)$ is a curve in the lower triangular unipotents with $(2,1)$ entry zero.
 \end{enumerate}
 
 We can now assume that the $C$-critical lattice $\Lambda$ has two independent points on $\partial C$ satisfying \eqref{3points}.
 The argument repeats almost identically taking $\vp_1, \vp_2$ to satisfy \eqref{3points} and $\vp_3$ satisfying \eqref{sides-of-cylinder}.
 {The only change in the course of this repetition is that the two eventualities in \eqref{change-for-n=2} become
 \begin{equation}
     \text{either }   {\pi_{+}}\big(N(\tau)\uu\big) < 1, \text{ or } N(\tau)\uu \in  \operatorname{span}_\Z \{\vp_1, \vp_2\}.
 \end{equation}
 And the second eventuality is again prevented by the definition of $N(\cdot)$ and the assumption that $N(\tau + \varepsilon) \uu \in C$.}
\end{proof}

\begin{lem}\label{cor-lemma1}
The lattice $\Lambda(1)$ in Lemma \ref{lemma1-CR} has at least one point satisfying
\begin{equation}\label{cor-to-1.1}
 \nu\big({\pi_{-}}(*)\big)<1, \ |{\pi_{+}}(*)| =1.   
\end{equation}
\end{lem}
\begin{proof}
This is true more generally. 
Let $\Lambda$ be $C$-critical.
If $\Lambda$ has no points satisfying %equation 
\eqref{cor-to-1.1}, 
%say $\Lambda := \Lambda(1)$ is generated by $\vq_1,\ \vq_2,\ \vq_3$. 
a contraction in the vertical direction would produce a $C$-admissible lattice with smaller covolume.
\end{proof}

We also need the following technical lemma from \cite[Lemma 2]{CR}.
Our statement here swaps out the strictly convex assumption for something that will be more convenient for our application.

\begin{lem}\label{Dr.Rado}
Let $B$ be a convex symmetric domain in $\R^2$. Suppose $\vp_1, \vp_2, \vp_3$ are three non-collinear points of $\overline{B}$ with at least one being in $B$.
Moreover, assume that the interior of a segment joining any two of the $\vp_i$ is in $B$.
Let $\Lambda$ be the lattice 
$$
\Z(\vp_2-\vp_1) + \Z(\vp_3-\vp_1).
$$
Then
$
\R^2 = \Lambda + B.
$
\end{lem}
\begin{proof}
Given any $\vq \in \R^2$ we can adjust by an element of $\Lambda$, to write $0$ as 
\begin{equation*}
    0 = \vq + u (\vp_2 - \vp_1)  + v (\vp_3 - \vp_1), \text{ with } 0 \leq u, v \leq 1.
\end{equation*}
Without loss of generality we can assume that $u$ and $v$ instead satisfy $0\leq u, v$  and $u + v \leq 1$. 
Indeed, if this is not the case, one can proceed with the argument by using $-\vq$ instead, and by noting that $B$ is symmetric.

Let $\vr = \vp_1 - \vq$.
If $\vr = 0$, then, since one of the $\vp_i$ is in $B$, we see that one of $$\vq, \vq + (\vp_2 - \vp_1), \vq + (\vp_3 - \vp_1)$$ is also in $B$.
If $\vr \neq 0$, we proceed by writing 
\begin{equation*}
    \vr = \vp_1 + u (\vp_2 - \vp_1) + v(\vp_3 - \vp_1) = (1 - u - v)\vp_1 + u \vp_2 + v \vp_3.
\end{equation*}
By the assumptions on $u, v$, this means that $\vr$ is in the convex hull of $\vp_1, \vp_2, \vp_3$.
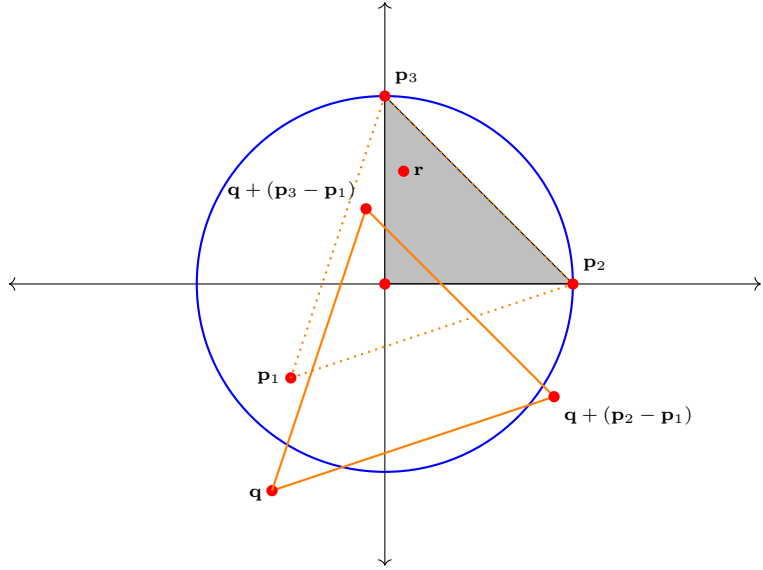
\begin{equation*}\label{diagram-Rado}
    \begin{tikzpicture}[scale=2.5, baseline=(current  bounding  box.center)]
    \draw [<->](-2,0) -- (2,0);
    \draw [<->] (0,-1.5) -- (0,1.5);
    
    \draw [thick, blue] (1,0) to [out=90, in=0] (0,1);
    \draw [thick, blue] (0,1) to [out=180, in=90] (-1,0);
    \draw [thick, blue] (-1,0) to [out=270, in=180] (0,-1);
    \draw [thick, blue] (0,-1) to [out=0, in=270] (1,0);
    
    \draw[fill=gray!50] plot[smooth, samples=100, domain=0:1] (\x,1- \x) -| (0,0) -- cycle;
    \fill [red] (0,0) circle[radius=0.03];
    
    \draw [thick, orange, dotted] (-0.5,-0.5) -- (1,0);
     \fill [red] (-0.5,-0.5) circle[radius=0.03];
    \node [left] at (-0.5, -0.5) {\tiny $\vp_1$};
    
    \draw [thick, orange, dotted] (1,0) -- (0,1);
    \fill [red] (1,0) circle[radius=0.03];
    \node [right] at (1,0.1) {\tiny $\vp_2$};
    
    \draw [thick, orange, dotted] (0,1) -- (-0.5,-0.5);
    \fill [red] (0,1) circle[radius=0.03];
    \node [right] at (0,1.1) {\tiny $\vp_3$};
    
    \draw [thick, orange] (-0.6,-1.1) -- (0.9,-0.6);
    \fill [red] (-0.6,-1-.1) circle[radius=0.03];
    \node [left] at (-0.6, -1.12) {\tiny $\vq$};
    
    \draw [thick, orange] (0.9,-0.6) -- (-0.1,0.4);
    \fill [red] (0.9,-0.6) circle[radius=0.03];
    \node [right] at (0.9,-0.7) {\tiny $\vq + (\vp_2-\vp_1)$};
    
    \draw [thick, orange] (-0.1,0.4) -- (-0.6,-1.1);
    \fill [red] (-0.1,0.4) circle[radius=0.03];
    \node [left] at (-0.1,0.5) {\tiny $\vq + (\vp_3-\vp_1)$};

    \fill [red] (0.1,0.6) circle[radius=0.03];
    \node [right] at (0.1,0.6) {\tiny $\vr$};

        %\node [below] at (0,-1.5) { \footnotesize The convex symmetric hexagon $H$ translated by $\vp_1$.};
    \end{tikzpicture}
\end{equation*}
Since the %convex hull of $\vp_1, \vp_2, \vp_3$ 
latter is contained in the union of convex hulls formed with any two of the $\vp_i$ along with $0$, we may write, using two distinct indices $i, j$, that 
\begin{equation*}\label{Rado-r}
    \vr = s \vp_i + t \vp_j \text{ where }  0\leq s , t \text{ and } 0 < s + t \leq 1.
\end{equation*}
Note the strict inequality which will be used later.
Using this along with the facts
\begin{equation*}
    \vq = \vp_1 - \vr,\ \vq + (\vp_2 - \vp_1) = \vp_2 - \vr,\ \vq + (\vp_3 - \vp_1) = \vp_3 - \vr,
\end{equation*}
we end up with the two containments
\begin{equation}\label{Rado-eqn}
    (1-s)\vp_i + t (-\vp_j) \in \vq + \Lambda\ \text{ and }\
    s (-\vp_i) + (1-t)\vp_j \in \vq + \Lambda.
\end{equation}
Since 
\begin{equation*}
    \big((1-s) + t\big) + \big(s + (1-t)\big) = 2,
\end{equation*}
one of the two points in %equation
 \eqref{Rado-eqn} must belong to the convex hull of some triple of points $\pm\vp_i, \mp\vp_j$ and $0$.
The conditions on $s, t$, the symmetry of $B$, and the assumption on the segments joining $\vp_i, \vp_j$ implies that the relevant point is in $B$ unless $s$ or $t$ is equal to $1$.
In the latter case, using the assumption that one of $\vp_1, \vp_2, \vp_3$ is in $B$, we can adjust either of the containments in %equation
 \eqref{Rado-eqn} by an element of $\Lambda$ to show that $\vq \in B + \Lambda$.
This completes the proof.
\end{proof}

At this point in \cite{CR}  the corresponding forms of Lemmata \ref{lemma1-CR} and \ref{Dr.Rado} are used to show %that $\Delta(B)=\Delta(C_B)$ 
\equ{equaldets} whenever $B$ is strictly convex.
Then, by approximating arbitrary convex domains by strictly convex ones, one gets the equality of critical determinants without assuming strict convexity. This approach however fails to provide the needed information on the critical locus of $C_B$. Thus
%However, in order to prove 
for the proof of Theorem \ref{critical-pieces} we need to make use of additional results of Mahler regarding convex symmetric domains in $\R^2$.

\begin{thm}[\cite{Ca}, \S V.8.3]\label{threepairs}
Let $\Lambda$ be $B$-critical, where $B\subset \R^2$ is a convex symmetric {bounded} domain.
Then one can find three pairs of points $\pm \vp, \pm \vq, \pm \vr$ of the lattice on $\partial B$.
%, the boundary of $B$. 
Moreover these three points can be chosen such that
\begin{equation}\label{inscirbedhexagon}
\vp = \vq - \vr
\end{equation}
and any two vectors among  $\vp,  \vq, \vr$ form a basis of $\Lambda.$

Conversely, if $\vp, \vq, \vr$ satisfying %equation
 {\eqref{inscirbedhexagon}} are on $\partial B$,
then the lattice generated by $\vp$ and $\vq$ is ${B}$-admissible.
Furthermore no additional (excluding the six above) point of $\Lambda$ is on $\partial B$ unless ${B}$ is a parallelogram. 
$\square$
\end{thm}
A convex symmetric {bounded} domain in $\R^n$ is said to be \textit{irreducible} if every properly contained convex symmetric domain has smaller critical determinant.
We make use of the following lemmata concerning irreducible domains in $\R^2$.
\begin{lem}[\cite{ma2}, Lemmata 5 and 9]\label{irredcritset}
Assume ${B}\subset \R^2$ is not a parallelogram and is irreducible.  {Then:
\begin{itemize}
\item[\rm (i)] for each $\vp \in %C = 
\partial {B}$ 
there is exactly one ${B}$-critical lattice %of ${B}$ 
containing $\vp$;
%Moreover, this critical lattice has exactly six points on the boundary of $K$.
\item[\rm (ii)] for each ${B}$-critical lattice $\Lambda$ and each $\vq,\vr \in %C = 
\partial {B}\cap \Lambda$, all points of the line segment between $\vq,\vr$ different from $\vq,\vr$ are interior points of $B$. $\square$
\end{itemize}}
\end{lem}
Non-parallelogram irreducible domains also have continuously differentiable boundaries.
\begin{lem}[\cite{ma3}, Theorem 3]\label{C1-boundary}
For $B\subset \R^2$ not a parallelogram and irreducible, the boundary $\partial B$ is a $\mathcal{C}^1$ submanifold of $\R^2$. $\square$
\end{lem}
Given two critical lattices of such a domain we have the following configuration of their points:
\begin{lem}[\cite{ma2}, Lemma $6$]\label{configuration-of-points}
{Suppose that} $B$ is not a parallelogram, $\Lambda$ is $B$-critical, and let $\vp_i: i = {1, \dots, 6}$ be the points of $\Lambda$ 
contained in  $\partial B$, labelled in a counter-clockwise order.
Let $A_i$ denote the open segment of $\partial B$ between $\vp_i$ and $\vp_{i+1}$.
If $\Lambda'$ is another $B$-critical lattice distinct from $\Lambda$, then each $A_i$ contains exactly one point of $\Lambda'$. $\square$
\end{lem}
\ignore{Fixing one critical lattice, this Lemma allows us to parameterize the critical locus of an irreducible $B$ by traversing along the boundary $\partial B$ from one lattice point to another (cf.\ \cite[Corollary 7.5]{KR}).
We also use the result of Mahler that every convex symmetric domain contains an irreducible domain with the same critical determinant:}
\begin{lem}[\cite{ma2}, Theorem 1]\label{contains-irred}
Every convex symmetric {bounded} domain $B\subset \R^2$ contains an irreducible $D$ with $\Delta(D) = \Delta(B).$
 $\square$
\end{lem}

\section{The critical locus of a cylinder: completing the proof of Theorem \ref{critical-pieces}}

We now have everything needed to deduce Theorem \ref{critical-pieces}.
\begin{proof}[Proof of Theorem \ref{critical-pieces}]
It is readily verified that each of the sets in \eqref{critical-piece-21} and \eqref{critical-piece-12} is in $\mathcal{L}(C)$.
We need to check the reverse inclusion.
Next, if $B$ is not irreducible, we can apply Lemma \ref{contains-irred} to obtain an irreducible ${D\subset B}$ with $\Delta(D) = \Delta(B)$.
More is true: since $B$ is not a parallelogram, a simple application of Minkowski's convex body theorem shows that $D$ is not a parallelogram either.
We have the inclusion $\mathcal{L}(B) \subset \mathcal{L}(D)$ and also the inclusion of the corresponding cylinders $C_D \subset C_B$, which again induces {the reverse} inclusion $\mathcal{L}(C_B)\subset \mathcal{L}(C_D)$.

{It is then elementary to observe that 
$\mathcal{L}(B)$ (resp., $\mathcal{L}(C_B)$) is exactly the set of $B$-admissible (resp., $C_B$-admissible) lattices in $\mathcal{L}(D)$ (resp., $\mathcal{L}(C_D)$),
see \cite[Lemma 3.1]{KRS}.
So assuming the theorem is true for $C_D$, it is easy to observe that the set of $C_B$-admissible lattices in $\mathcal{L}(C_D)$ is exactly what our theorem claims is $\mathcal{L}(C_B)$.
%\comm{This last part is left for the reader to verify. Is it ok?}
}
Thus we can (and will) assume for the remainder that $B$ is irreducible.
Let $\Lambda$ be $C$-critical.
We deal with two cases, the first being that when $\Lambda$ has three independent points satisfying equation \eqref{3points}.
In this case, we can assume that we have three independent points $\vp_1,\vp_2,\vp_3 \in \Lambda \cap \partial C$ all satisfying
\begin{equation}\label{mainthm=1}
    {\pi_{+}}(*)=1
\end{equation}
and, by Lemma \ref{cor-lemma1}, with $\vp_1$ satisfying %equation 
\eqref{cor-to-1.1}.

\begin{lem}
The points ${\pi_{-}}(\vp_1), {\pi_{-}}(\vp_2), {\pi_{-}}(\vp_3)$ satisfy the hypotheses of Lemma \ref{Dr.Rado}.
\end{lem}
\begin{proof}
The assumptions on the {points} $\vp_i$ ensure that the ${\pi_{-}}(\vp_i)$ are in $\overline{B}$.
These assumptions also ensure that ${\pi_{-}}(\vp_1)$ is in $B$.
It remains to show that the open segment joining ${\pi_{-}}(\vp_2)$ and ${\pi_{-}}(\vp_3)$ is contained in $B$.
Since $B$ is irreducible, we see from Lemma \ref{irredcritset} that the only way this fails is if we have a $B$-critical lattice $\Lambda'$ with two points $\vq_1, \vq_2 \in \Lambda' \cap \partial B$, each distinct from ${\pi_{-}}(\vp_2)$ and ${\pi_{-}}(\vp_3)$, such that one of the segments of $\partial B$ connecting $\vq_1$ to $\vq_2$ contains both ${\pi_{-}}(\vp_2)$ and ${\pi_{-}}(\vp_3)$ but no other point of $\Lambda'$.

In light of Theorem \ref{threepairs}, we see that the six points of $\Lambda' \cap \partial B$ are
\begin{equation}\label{hexagon-vertices}
    \vq_1,\ \vq_2,\ \vq_2-\vq_1,\ -\vq_1,\ -\vq_2,\ \vq_1 - \vq_2.
\end{equation}
Since $\vq_1 - \vq_2$, $\vq_1$, $\vq_1 + \vq_2$ are collinear and similarly with $\vq_1$ and $\vq_2$ swapped, the segment of $\partial B$ joining $\vq_1$ and $\vq_2$ must lie in the convex hull, call it $T$, of $\vq_1$, $\vq_2$ and $\vq_1+\vq_2$.
In particular, using the assumptions on $\Lambda'$ along with Lemma \ref{irredcritset},
both ${\pi_{-}}(\vp_2)$ and ${\pi_{-}}(\vp_3)$ must belong to the interior of $T$.
See the diagram \eqref{triangle} below for clarity.

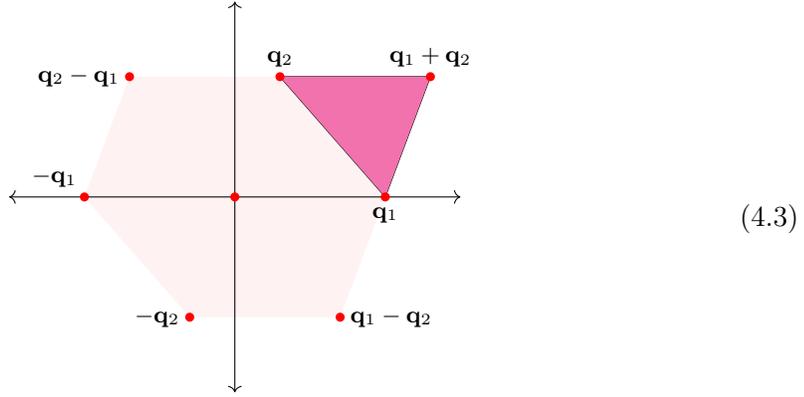
\begin{equation}\label{triangle}
    \begin{tikzpicture}[scale=2,  baseline=(current  bounding  box.center)]
    
    \fill[pink!20] (1,0) -- (0.3, 0.8) -- (-0.7, 0.8) -- (-1,0) -- 
    (-0.3, -0.8) -- (0.7, -0.8) -- cycle;
    
    \draw [<->](-1.5,0) -- (1.5,0);
    \draw [<->] (0,-1.3) -- (0,1.3);
    
    \draw (1,0) -- (0.3,0.8);
    \draw (1.3,0.8) -- (0.3,0.8);
    \draw (1,0) -- (1.3,0.8);
    
    \fill[magenta!70] (1,0) -- (1.3,0.8) -- (0.3,0.8) -- cycle;

    \fill [red] (1,0) circle[radius=0.03];
    \fill [red] (0.3,0.8) circle[radius=0.03];
    \fill [red] (-0.7,0.8) circle[radius=0.03];
    \fill [red] (0,0) circle[radius=0.03];
    \fill [red] (1.3,0.8) circle[radius=0.03];
    \fill [red] (-1,0) circle[radius=0.03];
    \fill [red] (-0.3,-0.8) circle[radius=0.03];
    \fill [red] (0.7,-0.8) circle[radius=0.03];
    %\fill [red] (-0.4,1.6) circle[radius=0.03];
    
    \node [below] at (1,0) {\footnotesize $\vq_1$};
    \node [above] at (0.3,0.8) {\footnotesize $\vq_2$};
    \node [above] at (1.3, 0.8) {\footnotesize $\vq_1+\vq_2$};
    \node [left] at (-0.7, 0.8) {\footnotesize $\vq_2-\vq_1$};
    \node [above] at (-1.2, 0) {\footnotesize $-\vq_1$};
    \node [left] at (-0.3, -0.8) {\footnotesize $-\vq_2$};
    \node [right] at (0.7, -0.8) {\footnotesize $\vq_1-\vq_2$};
    \node [below] at (0,-1.4) {\footnotesize The sum $T+(-T)$ of the magenta triangle is %\comm{\sout{contained in}}
    the pink hexagon.};
    %\node [below] at (0,-1.83) {\footnotesize triangle.};

    \end{tikzpicture}
\end{equation}
The Minkowski sum $T+(-T)$ is %a compact convex set. 
%Moreover, the extreme points of $T+(-T)$ are clearly contained in 
%\new
{the convex hull of} the set of pairwise differences of the %extreme points 
%\new
{vertices} of $T$
(that is, the differences of the points $\vq_1$, $\vq_2$ and $\vq_1+\vq_2$)
%By the Krein--Milman theorem (see \cite[Theorem 3.23]{Ru}) $T+(-T)$ is contained in the convex hull of the six points %of equation
%in  \eqref{hexagon-vertices} \comm{(I couldn't prove the sum is exactly the hexagon. Any suggestions?)}, 
which is contained in $\overline{B}$.
Moreover, ${\pi_{-}}(\vp_2) - {\pi_{-}}(\vp_3)$ actually belongs to the convex open set $\operatorname{interior}(T) + \big(- \operatorname{interior}(T)\big)$, which in turn is contained in $B$.

Since $\vp_2$ and $\vp_3$ are in $\Lambda$, and since they both satisfy equation \eqref{mainthm=1}, 
the above discussion shows that the open segment joining ${\pi_{-}}(\vp_2)$ and ${\pi_{-}}(\vp_3)$ cannot contain points of $\partial B$ without contradicting the $C$-admissibility of $\Lambda$.
Thus the lemma is proved and we may apply Lemma \ref{Dr.Rado}.
\end{proof}

Given any $\vp \in \Lambda$, there exists $u\in \Z$ such that the point $\vp - u\vp_1$ satisfies
\begin{equation}\label{maintheorem<1/2}
      |{\pi_{+}}(*)| \leq 1/2.
\end{equation}
By {applying} Lemma \ref{Dr.Rado} to the three points ${\pi_{-}}(\vp_i)$, we see that there exist $v, w \in \Z$ such that 
\begin{equation*}
    {\pi_{-}}(\vp - u\vp_1) - v\big({\pi_{-}}(\vp_2)-{\pi_{-}}(\vp_1)\big) - w\big({\pi_{-}}(\vp_3) - {\pi_{-}}(\vp_1)\big) \in B.
\end{equation*}
Since $\Lambda$ is $C$-admissible, this along with %equation
 \eqref{maintheorem<1/2} shows that we actually have
\begin{equation*}
    \vp - u\vp_1 - v(\vp_2-\vp_1) - w(\vp_3 - \vp_1) = 0.
\end{equation*}
Thus we have
\begin{equation*}
    \Lambda =  
  \left[ {\begin{array}{ccc}
   \vp_2 - \vp_1 & \vp_3 - \vp_1 & \vp_1 \\
  \end{array} } \right] \Z^3.
\end{equation*}
Equation \eqref{mainthm=1} says that all the points $\vp_i$ are on top of the cylinder, so that we are in the case where
\begin{equation*}
    \Lambda = \left[ {\begin{array}{ccc}
   1 & 0 & * \\
   0 & 1 & * \\
   0 & 0 & 1
  \end{array} } \right]
    \left[ \begin{array}{c | c} 
      M &  \begin{array}{c} 0 \\ 0 \end{array} \\ 
      \hline
      \begin{array}{c c} 0 & 0 \end{array} & 1 
     \end{array} \right] \Z^3
\end{equation*}
with $M\Z^2$ being $B$-admissible.
On the other hand, since $\Delta(C) \leq \Delta(B)$, we must have $d(M\Z^2) \leq \Delta(B)$.
From the definition of $\Delta(B)$ we get that $d(M\Z^2) = \Delta(B)$ so that it is actually $B$-critical.
Thus we have shown $\Lambda$ is in the set \eqref{critical-piece-12}.

We now turn to the case where $\Lambda$ has fewer than three independent points satisfying %equation
 \eqref{3points}.
$B$ is still assumed to be irreducible.
As in Lemma \ref{lemma1-CR}, we construct a path of $C$-critical lattices
$$
L(t)\Lambda,\ \ t\in [0,t_1]
$$
of the form \eqref{parameterization-of-continuous-family}.
Precisely, $L(t)$ is lower triangular unimodular with $(2,1)$-entry zero
and $t_1$ {is such} that 
$L(t_1)\Lambda$ has three {linearly} independent points satisfying %equation 
\eqref{3points}, but no such triple exists for $t<t_1$.
We analyze the lattice $L(t_1)\Lambda$ and the configuration of points lying on the top of the cylinder. 

To fix notation, let 
\begin{equation}\label{ab}
L(t) = 
\left[ {\begin{array}{ccc}
   1 & 0 & 0 \\
   0 & 1 & 0 \\
   t\alpha & t\beta & 1
  \end{array} } \right].
\end{equation}
{Now suppose} %say 
we have the points $\vp(t) \in L(t)\Lambda \cap \partial C$ with 
\begin{equation*}
{\pi_{+}}\big(\vp(t)\big) < 1 \text{ for $t<t_1$, and } {\pi_{+}}\big(\vp(t_1)\big) = 1. 
\end{equation*}
We can assume $\vp(0) = (p_1, p_2, p_3) \in \Lambda$, and $\vp(t) = L(t)\vp(0)$.
Precisely, 
\begin{equation*}
    \vp(t) = \big(p_1,\ p_2,\ t(\alpha p_1 + \beta p_2) + p_3\big).
\end{equation*}
For later use we record here that
\begin{equation*}\label{crucial}
    (p_1, p_2)\cdot(\alpha, \beta) :=(\alpha p_1 + \beta p_2) > 0.
\end{equation*}

Now take three {linearly} independent points $\vp_1$, $\vp_2$, $\vp_3$ of $L(t_1)\Lambda$ on top of the cylinder while fixing $\vp_1:= \vp(t_1)$.
Note that by Lemma \ref{cor-lemma1} we can assume that $\vp_2$ satisfies 
\begin{equation*}
\nu\big({\pi_{-}}(*)\big) <1.
\end{equation*}
Applying Lemma \ref{Dr.Rado} exactly as in first case dealt with above, we arrive at the conclusion that lattice in $\R^2$ generated by ${\pi_{-}}(\vp_2 - \vp_1)$ and ${\pi_{-}}(\vp_3 -\vp_1)$ is $B$-critical.
Thus $L(t_1)\Lambda$ is again of the form
\begin{equation*}
    L(t_1)\Lambda = \left[ \begin{array}{c | c} 
      M &  \begin{array}{c} * \\ * \end{array} \\ 
      \hline
      \begin{array}{c c} 0 & 0 \end{array} & 1 
     \end{array} \right] \Z^3
\end{equation*}
with $M\Z^2$ being $B$-critical.
Our goal now is to show that $L(t_1)\Lambda$ contains the point $(0,0,1)$.
This would enable us to write
\begin{equation*}
    L(t_1)\Lambda = \left[ \begin{array}{c | c} 
      M &  \begin{array}{c} 0 \\ 0 \end{array} \\ 
      \hline
      \begin{array}{c c} 0 & 0 \end{array} & 1 
     \end{array} \right] \Z^3, \text{ or rather } 
     \Lambda = \left[ {\begin{array}{ccc}
   1 & 0 & 0 \\
   0 & 1 & 0 \\
   * & * & 1
  \end{array} } \right]
  \left[ \begin{array}{c | c} 
      M &  \begin{array}{c} 0 \\ 0 \end{array} \\ 
      \hline
      \begin{array}{c c} 0 & 0 \end{array} & 1 
     \end{array} \right] \Z^3,
\end{equation*}
which would show that $\Lambda$ is in the set \eqref{critical-piece-21}.

To this end, we need to take careful stock of the various moving parts occurring in the deformation of $\Lambda$.
The reader is advised to stare at diagram \eqref{diagram-disc} for a minute before reading on.

From Lemma \ref{C1-boundary}, let $T$ be the tangent line to $\big(B + {\pi_{-}}(\vp_1)\big)$ at zero.
Note that, by Theorem \ref{threepairs} and Lemma \ref{irredcritset}, the intersection
\begin{equation*}
    \partial B \cap \big({\pi_{-}}(\vp_1) + \partial B\big)
\end{equation*}
consists of exactly two points $\vr_1$ and $\vr_2$ which generate the unique critical lattice of $B$ containing ${\pi_{-}}(\vp_1)$.
Let $A_1$ and $A_2$ denote the open arcs of ${\pi_{-}}(\vp_1) + \partial B$ joining the origin and $\vr_1$ and $\vr_2$ respectively.
By Lemma \ref{configuration-of-points}, any $B$-critical lattice distinct from the one containing ${\pi_{-}}(\vp_1)$ is determined by its intersection with $A_1 - {\pi_{-}}(\vp_1)$ and uniquely determines a second point in $A_2 - {\pi_{-}}(\vp_1)$.
Moreover, by Lemma \ref{irredcritset}, these two hypothesized points in $A_1 - {\pi_{-}}(\vp_1)$ and $A_2 - {\pi_{-}}(\vp_1)$ cannot both lie on the tangent $T -{\pi_{-}}(\vp_1)$ to $B$ at $-{\pi_{-}}(\vp_1)$.

All of this information is neatly summarized in the diagram below.
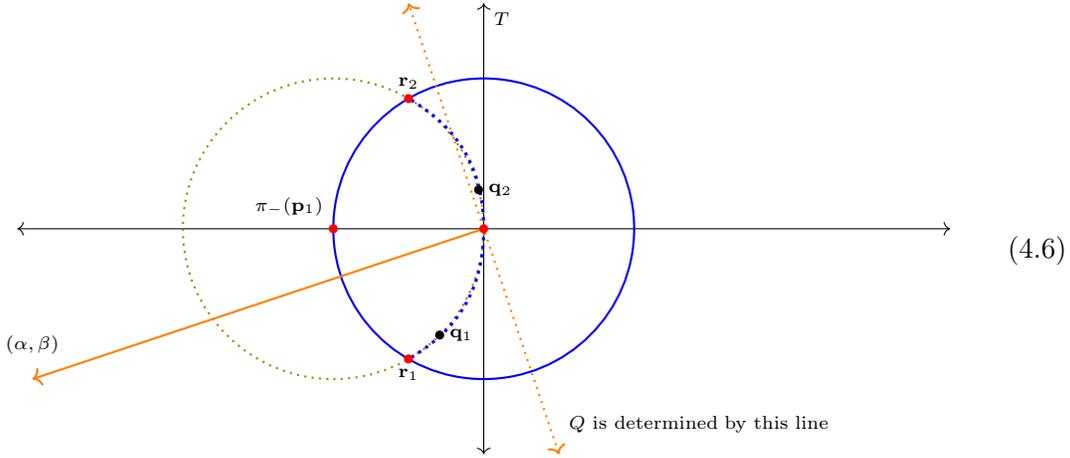
\begin{equation}\label{diagram-disc}
    \begin{tikzpicture}[scale=2, baseline=(current  bounding  box.center)]
    \draw [<->](-3.1,0) -- (3.1,0);
    \draw [<->] (0,-1.5) -- (0,1.5);
    
    \node[right] at (0,1.4) {\tiny $T$};
    
    \draw [thick, blue] (1,0) to [out=90, in=0] (0,1);
    \draw [thick, blue] (0,1) to [out=180, in=90] (-1,0);
    \draw [thick, blue] (-1,0) to [out=270, in=180] (0,-1);
    \draw [thick, blue] (0,-1) to [out=0, in=270] (1,0);
    
    \draw [thick, olive, dotted] (0,0) to [out=90, in=0] (-1,1);
    \draw [thick, olive, dotted] (-1,1) to [out=180, in=90] (-2,0);
    \draw [thick, olive, dotted] (-2,0) to [out=270, in=180] (-1,-1);
    \draw [thick, olive, dotted] (-1,-1) to [out=0, in=270] (0,0);
    %\draw [semithick, blue] (0.7,0.6) to [out=120, in=0] (0,1);
    
    \draw [very thick, blue, dotted]  (0,0) to [out=90, in=330] (-0.5,0.866);
    \draw [very thick, blue, dotted]  (0,0) to [out=270, in=30] (-0.5,-0.866);
    
    \draw [thick, orange, ->] (0,0) to (-3,-1);
    \node [above] at (-3,-0.9) {\tiny $(\alpha, \beta)$};
    \draw [thick, orange, dotted, <->] (0.5,-1.5) to (-0.5,1.5);
    \node [right] at (0.5,-1.3) {\tiny $Q$ is determined by this line};    

    \fill [red] (0,0) circle[radius=0.03];
    
    \fill [red] (-0.5,0.866) circle[radius=0.03];
    \node [above] at (-0.5,0.866) {\tiny $\vr_2$};
    
    \fill [red] (-0.5,-0.866) circle[radius=0.03];
    \node [below] at (-0.5,-0.866) {\tiny $\vr_1$};
    
    \fill [red] (-1,0) circle[radius=0.03];
    \node [left] at (-1,0.15) {\tiny ${\pi_{-}}(\vp_1)$};
    
    \fill [black] (0.707-1, -0.707) circle[radius=0.03];
    \node [right] at (0.707-1, -0.707) {\tiny $\vq_1$};
    
    \fill [black] (0.966-1, 0.259) circle[radius=0.03];
    \node [right] at (0.966-1,0.259) {\tiny $\vq_2$};
    
        \node [below] at (0,-1.6) { \footnotesize Are the $\vq_i$ above a possible configuration for the points ${\pi_{-}}(\vp_2)$ and ${\pi_{-}}(\vp_3)$?};
    %\node [below] at (0,-1.83) {\footnotesize All other critical lattices are obtained by rotating this one.};
    \end{tikzpicture}
\end{equation}
We can finally prove
\begin{lem}
Write $\Lambda'$ to be the $B$-critical lattice
\begin{equation*}
    \Z{\pi_{-}}(\vp_2-\vp_1) + \Z{\pi_{-}}(\vp_3-\vp_1).
\end{equation*}
Then the origin belongs to the grid 
\begin{equation}\label{grid}
    {\pi_{-}}(\vp_1) + \Lambda'.
\end{equation}
\end{lem}
\begin{proof}
There are two possibilities for the $B$-critical lattice $\Lambda'$.
It is either the $B$-critical lattice containing ${\pi_{-}}(\vp_1)$, or %otherwise, 
it contains a point on each of the open arcs $A_i - {\pi_{-}}(\vp_1)$.
The conclusion of the lemma holds in the first case.

So, for sake of contradiction, we assume the second.
Let $\vq_1- {\pi_{-}}(\vp_1)$ and $\vq_2 - {\pi_{-}}(\vp_1)$ denote the two points of $\Lambda'$ on the open arcs $A_i - {\pi_{-}}(\vp_1)$.
Then $\vq_1$ and $\vq_2$ are in $A_1$ and $A_2$ respectively.
Moreover, these points are in the grid \eqref{grid} and hence are the images of points in the $C$-critical {lattice} $L(t_1)\Lambda$ that satisfy
\begin{equation*}
    \nu\big({\pi_{-}}(*)\big) < 1 \text{ and } {\pi_{+}}(*) = 1.
\end{equation*}
Now the vector $(\alpha, \beta)$ occurring in the deformation %of equation
 \eqref{ab} determines an open half-plane $Q$ by the inner-product condition
\begin{equation*}\label{p1-product}
    (\alpha, \beta)\cdot (*, *) > 0.
\end{equation*}
Since $L(t)\Lambda$ is $C$-admissible for $0<t<t_1$, we see that both $\vq_1$ and $\vq_2$ must satisfy
\begin{equation*}\label{qi-product}
    (\alpha, \beta)\cdot (*,*) \leq 0.
\end{equation*}
That is, ${\pi_{-}}(\vp_1) \in Q$, and $\vq_1$ and $\vq_2$ must belong to the closed half-plane $Q^c$.

Let $A$ be the {set of} boundary points of $\partial B + {\pi_{-}}(\vp_1)$ contained in $Q^c$. We have just seen that $\vq_1$, $\vq_2$ and $0$ are in $A$, whereas $2{\pi_{-}}(\vp_1)$ is not.
Since $0$ is in the boundary of $Q$, and since $\vq_1$ and $\vq_2$ are in different components of the complement of the line determined by ${\pi_{-}}(\vp_1)$, we have that $A$ must actually be a straight line.

This is a contradiction to Lemma \ref{irredcritset} since, as noted above, $\vq_1$ and $\vq_2$ cannot both belong to the tangent {line} $T$.
Thus we necessarily have that ${\pi_{-}}(\vp_1) \in \Lambda'$, and the conclusion of the lemma holds.
\end{proof}
From this lemma and the fact that each of the $\vp_i$ is on top of the cylinder, that is, satisfies equation \eqref{mainthm=1}, we see that $(0,0,1)$ does indeed belong to $L(t_1)\Lambda$.
This concludes this final case of the theorem.
\end{proof}

\begin{exmp}\label{euclid} \rm
%As an example, 
The critical locus of the cylinder over the unit disc is given by the union of two compact submanifolds (cf.\ \cite{Ma1}):
\begin{equation}\label{K-euclidean-1}
    \left\lbrace
        \left[ {\begin{array}{ccc}
   \cos(t) & -\sin(t) & 0 \\
   \sin(t) & \cos(t) & 0 \\
   x & y & 1
  \end{array} } \right]
        \left[ {\begin{array}{ccc}
   1 & 1/2 & 0 \\
   0 & \sqrt{3}/2 & 0 \\
   0 & 0 & 1
  \end{array} } \right] \Z^3 : x,y,t \in \R
     \right\rbrace 
\end{equation}
and
\begin{equation}\label{K-euclidean-2}
    \left\lbrace
        \left[ {\begin{array}{ccc}
   \cos(t) & -\sin(t) & x \\
   \sin(t) & \cos(t) & y \\
   0 & 0 & 1
  \end{array} } \right]
        \left[ {\begin{array}{ccc}
   1 & 1/2 & 0 \\
   0 & \sqrt{3}/2 & 0 \\
   0 & 0 & 1
  \end{array} } \right] \Z^3 : x,y,t \in \R
     \right\rbrace.
\end{equation}
Each is the orbit of a subgroup of $\operatorname{SL}_3(\R)$, and the same is true for their intersection.
{Letting $\nu$ stand for the Euclidean norm in $\R^2$, on scaling each lattice above by $$\Delta_\nu^{-1/3} = \left(\sqrt{3}/2\right)^{-1/3},$$ we get the critical locus $\mathcal{L}_\eta$ associated to the %Euclidean 
norm \equ{norm}.
As mentioned before, this normalization is done to ensure that the critical locus consists of unimodular lattices.}\end{exmp}

\section{The interaction of the critical locus and the flow}\label{completion}

{In order to use the results of the previous sections for the proof of our main result, it will be convenient to introduce the following two subsets of $X_3$:
\begin{equation*}\label{zplus}
  \mathcal{Z}_+ := \left\lbrace
        \left[ {\begin{array}{ccc}
   * & * & 0\\
   * & * & 0 \\
   * & * & *
  \end{array} } \right]\Z^3
     \right\rbrace
\quad\text{and}\quad
%\begin{equation}\label{zminus}
  \mathcal{Z}_- := \left\lbrace
        \left[ {\begin{array}{ccc}
   * & * & * \\
   * & * & * \\
   0 & * & *
  \end{array} } \right]\Z^3
     \right\rbrace   . 
\end{equation*}
In words, $\mathcal{Z}_+$  (resp., $\mathcal{Z}_-$) consist of  lattices having a nonzero vector on the $x_3$-axis (resp., on the $(x_1,x_2)$-plane).}

\smallskip

Now observe that a combination of Theorems \ref{critical-pieces} and \ref{Hajos} readily produces the following 
\begin{prop}\label{critical-corollary}
Let %$\nu$ be a norm on $\R^2$, and let 
  $\eta$ be the norm on $\R^3$ of the form \equ{norm}. Then its critical locus $\mathcal{L}_\eta$ is contained in  $\mathcal{Z}_+\cup\mathcal{Z}_-$. 
\end{prop}
\begin{proof} 
Let $B:=B_\nu\left(%\frac
{1}/{\Delta_\nu^{1/3}}\right)$, %\subset \R^2$$ is not a parallelogram.
%As above $\eta$ denotes the cylindrical norm associated to $\nu$, and 
and recall that we have the equality $\Delta_\eta = \Delta_\nu$.
Note   that if $g \in \GL_3(\R)$ %has positive determinant 
and if $D$ is any domain, we have
\begin{equation}\label{transf}
    \Delta(gD) = |\det(g)|\Delta(D) \text{ and } \mathcal{L}(gD) = g \mathcal{L}(D).
\end{equation}
Using the matrix
\begin{equation}\label{g}
    g = \left[ {\begin{array}{ccc}
   1 & 0 & 0 \\
   0 & 1 & 0 \\
   0 & 0 & \Delta_\nu^{1/3}
  \end{array} } \right],
\end{equation}
we see that 
\begin{equation*}
    gB_\eta\left(%\frac
    {1}/{\Delta_\eta^{1/3}}\right) = B  \times [-1,1].
\end{equation*}
Suppose that $B$ is not a parallelogram. Then by Theorem \ref{critical-pieces}, %the critical locus $\mathcal{L}_\eta$ is the union of the two sets
\begin{equation*}
\mathcal{L}_\eta =    %Z_1 =
 \left\lbrace
        \left[ \begin{array}{c | c} 
      M &  \begin{array}{c} 0 \\ 0 \end{array} \\ 
      \hline
      \begin{array}{c c} * & * \end{array} & \Delta_\nu^{-1/3} 
     \end{array} \right] \Z^3 : M\Z^2 \in \mathcal{L}(B)
     \right\rbrace
\bigcup  % Z_2 =
 \left\lbrace
        \left[ \begin{array}{c | c} 
      M &  \begin{array}{c} * \\ * \end{array} \\ 
      \hline
      \begin{array}{c c} 0 & 0 \end{array} & \Delta_\nu^{-1/3} 
     \end{array} \right] \Z^3 : M\Z^2 \in \mathcal{L}(B)
     \right\rbrace.
\end{equation*}
This finishes the proof in this case, since the sets in the right hand side of the above equation are contained in $\mathcal{Z}_+$  and $\mathcal{Z}_-$ respectively.

In the remaining parallelogram  case we have $B_\eta\left(%\frac
    {1}/{\Delta_\eta^{1/3}}\right) = gC$, where $g$ is a linear transformation of the form  
$ \left[ \begin{array}{c | c} 
      * &  \begin{array}{c} 0 \\ 0 \end{array} \\ 
      \hline
      \begin{array}{c c} 0 & 0 \end{array} & * 
     \end{array} \right] $,  and $C$ is a ball for the supremum norm in $\R^3$.  Theorem \ref{Hajos} implies that every lattice in $\mathcal{L}(C)$ contains at least one vector from the standard basis of $\R^3$; hence $\mathcal{L}(C)\subset \mathcal{Z}_+\cup\mathcal{Z}_-$. Since $g$ leaves both the $(x_1,x_2)$-plane and  the $x_3$-axis invariant, it also preserves the sets $\mathcal{Z}_+$  and $\mathcal{Z}_-$, which, in view of \eqref{transf}, finishes the proof of the proposition.
  \end{proof}

We are now ready for the
\begin{proof}[Proof of Theorem \ref{main-theorem}] 
Recall   that $\x\in\R^2$ is badly approximable if and only if the orbit $\{a_s\Lambda_\x: s>0\}$ is bounded (see \cite[Theorem 2.20]{d}).
 Take $\x\in\mathbf{BA}$ and let $\mathcal{K}\subset X_3$ be a compact set containing $\{a_s\Lambda_\x : s\ge 0\}$. In view of Corollary \ref{miss-locus}, we need to prove that the forward $a_s$-orbit of $\Lambda_\x$ eventually avoids an open neighborhood of $\mathcal{L}_\eta$. Suppose that it is not the case, that is, we have  $a_{s_k}\Lambda_\x\to\Lambda\in\mathcal{L}_\nu$ for some sequence $s_k\to\infty$. In view of the preceding proposition we are left to consider two cases.
\begin{itemize}
\item[\bf Case 1.] $\Lambda\in\mathcal Z_+$. Since the $x_3$-axis is a contracting eigenspace for $a_s$ for $s > 0$, it follows that the forward  $a_s$-trajectory of $\Lambda$ is divergent, which implies that
\begin{equation}\label{forward}
    \{s > 0 : a_s \Lambda \in \mathcal{K}\} \text{ is bounded.}
\end{equation}
By continuity of the action and in view of \eqref{forward}, there is a $k\in\N$ and an open neighborhood $ \mathcal{V}$ of $\Lambda$ such that $a_{s_k}\mathcal{V} \cap  \mathcal{K} =\varnothing$. 
This implies that $a_{s_k+s_m}\Lambda_\x \notin \mathcal{K}$ for large enough $m$, a contradiction.
\item[\bf Case 2.] $\Lambda\in\mathcal Z_-$. In this case, since  the $(x_1,x_2)$-plane is a contracting eigenspace for $a_s$ for $s < 0$,  the backward  $a_s$-trajectory of $\Lambda$ is divergent,  That is, \begin{equation}\label{backward}
    \{s < 0 : a_s \Lambda \in \mathcal{K}\} \text{ is bounded.}
\end{equation}
Again by continuity of the action and in view of \eqref{backward}, there is a neighborhood $ \mathcal{V}$ of $\Lambda$   and $s<0$ such that $a_s \mathcal{V} \cap  \mathcal{K} =\varnothing$.
This implies that there   exist infinitely many $k$ for which $a_{s+s_k}\Lambda_\x \notin  \mathcal{K} $.
The flow time $s+s_k$ will eventually be positive, contradicting the fact that the forward orbit of $\Lambda_\x$ is contained in $K$.
\end{itemize}
\end{proof}

\begin{proof}[Proof of Corollary \ref{cor-spectrum}]
{Suppose, to the contrary, that $c_\nu$ is an isolated point of $\D_\nu$. This means that there exists $c < c_\nu$ such that $ c_\nu(\x) < c$ for all $\x\in \mathbf{DI}_\nu$. Proposition \ref
{dynamical} then implies that there exists $r < r_\eta$, where $\eta$ is again defined via \equ{norm}, such that
\eq{containment}{\mathbf{DI}_\nu\subset \left\{\x\in\R^2: a_su_\x\cdot\Z^3 \notin  \mathcal{K}_\eta(r)\text{ for all sufficiently large }s\right\}.
}
The exceptional sets similar to that in the right hand side of \equ{containment} were recently investigated in \cite{KM}, where, in particular, the so-called Dimension Drop Conjecture was solved for the $a_s$-action on $X_3$. More precisely, 
since $\mathcal{K}_\eta(r)$ has non-empty interior, in view of \cite[Theorem~1.2]{KM}
the set in the right hand side of \equ{containment} has less than full \hd, which contradicts Theorem \ref{main-theorem}.}  \end{proof}

{We end the paper with a few remarks and questions.}

\begin{rem}{%It is desirable to obtain 
\rm Let $B \subset \R^n$ be a bounded convex symmetric domain, and let $C_B\subset \R^{n+1}$ be the cylinder over it.
It would be natural to attempt to obtain a structure theorem for $\mathcal{L}(C_B)$ %when $B \subset \R^n$ and $C_B\subset \R^{n+1}$ is the cylinder over it.
similar to Theorem \ref{critical-pieces}.
However, even the equality of critical determinants
%the statement 
%
%that $\Delta(B) = \Delta(C_B)$
 is not known in this generality.
The only higher-dimensional example known to the authors can be found in \cite{W}: it  establishes \equ{equaldets} for the special case of %the Euclidean norm
$B$ being a Euclidean ball in $\R^3$.} \end{rem}

\begin{rem}{\rm Corollary \ref{miss-locus} 
suggests that one can consider a \di\ condition on $\x\in\R^{2}$ equivalent to  the forward $a_s$-orbit of $\Lambda_\x$ eventually avoiding some open neighborhood of $\mathcal{L}_\eta$,
where $\eta$ is an arbitrary (not necessarily cylindrical) norm on $\R^3$. This variation on the theme of improving Dirichlet's theorem has been thoroughly explored in \cite{KR} in a much more general set-up of $m\times n$ matrices. In particular, the full Hausdorff dimension (and, even stronger, the hyperplane absolute winning property) of the sets of lattices avoiding the critical locus has been established in several special cases, such as: an arbitrary norm on $\R^2$ \cite[Theorem 1.3]{KR} and the Euclidean norm on $\R^{m+n}$ \cite[Theorem 3.7]{KR}. In both cases a connection with badly approximable matrices   is no longer clear, and the proof is based on results from \cite{AGK} of the first named author with An and Guan. Whether or not a similar full dimension or winning result holds for an arbitrary norm on $\R^{m+n}$  is an open question.
}\end{rem}

\end{document}